\documentclass[a4paper,12pt]{amsart}
\usepackage{graphicx,tabularx} 
\usepackage[left=2cm,top=2.5cm,right=2cm,bottom=2.5cm]{geometry}
\usepackage{amsmath}
\usepackage{amssymb}
\usepackage{amsfonts}
\usepackage{amsthm}
\usepackage{stmaryrd}
\usepackage{url}
\usepackage{hyperref}
\usepackage{fancyhdr}
\usepackage{mathrsfs}
\usepackage{mathtools}
\usepackage{color,colortbl}
\usepackage{enumerate,mdwlist}
\usepackage{tikz-cd} 
\usepackage[T1]{fontenc}

\usepackage{txfonts,pxfonts}

\usepackage{amstext}
\usepackage{array}

\usepackage[shortlabels]{enumitem}
\setlist[enumerate]{leftmargin=25pt}
\setlist[itemize]{leftmargin=25pt}

\newtheorem{thm}{Theorem}[section]
\newtheorem{lemma}[thm]{Lemma}
\newtheorem{prop}[thm]{Proposition}
\newtheorem{cor}[thm]{Corollary}
\newtheorem{conj}[thm]{Conjecture}

\theoremstyle{definition}

\newtheorem{claim}{Claim}

\theoremstyle{definition}
\newtheorem{defn}[thm]{Definition}

\newcommand{\di}{\ensuremath{\mathrm{d}}}

\newcommand{\bra}[1]{{\left({#1}\right)}}

\newcommand{\set}[1]{{\left\{{#1}\right\}}}

\DeclarePairedDelimiter\abs{\lvert}{\rvert}
\DeclarePairedDelimiter\norm{\lVert}{\rVert}

\newcommand{\al}{\ensuremath{\alpha}}

\newcommand{\de}{\ensuremath{\delta}}

\newcommand{\eps}{\ensuremath{\varepsilon}}

\newcommand{\La}{\ensuremath{\Lambda}}
\newcommand{\la}{\ensuremath{\lambda}}

\newcommand{\vphi}{\ensuremath{\varphi}}
\newcommand{\om}{\ensuremath{\omega}}
\newcommand{\Om}{\ensuremath{\Omega}}

\newcommand{\mc}[1]{\ensuremath{\mathcal{#1}}}

\newcommand{\ZZ}{\ensuremath{\mathbb{Z}}}

\newcommand{\QQ}{\ensuremath{\mathbb{Q}}}
\newcommand{\RR}{\ensuremath{\mathbb{R}}}
\newcommand{\NN}{\ensuremath{\mathbb{N}}}
\newcommand{\TT}{\ensuremath{\mathbb{T}}}

\newcommand{\CC}{\ensuremath{\mathbb{C}}}

\newcommand{\sm}{\ensuremath{\setminus}}

\newcommand{\vn}{\ensuremath{\varnothing}}

\usepackage[normalem]{ulem}

\newcommand{\out}[1]{}

\numberwithin{equation}{section}

\title{Linear Independence of Time-Frequency Translates for Ultimately Positive Functions}
\author{Romanos Diogenes Malikiosis and Nikos Poursalidis}

\begin{document}

\begin{abstract}
 We prove that the HRT conjecture \cite{HRT} holds when the Gabor system consists of a 4-point set in the time-frequency plane and a square-integrable function that is ultimately positive. We also prove the conjecture for Gabor systems generated by an ultimately positive function and translation sets, whose frequencies satisfy at most one linear dependence over $\mathbb{Z}$, improving a result of Benedetto and Bourouihiya \cite{infbehavior}.
\end{abstract}

\maketitle

\section{Introduction}

Let $f \in L^2(\mathbb{R})$ and $ \Lambda \subset \mathbb{R}^2  $ be a finite set of points. The \textit{Gabor system} generated by $f$ and $\Lambda$ is the set of time-frequency shifts:
$$ \mathcal{G}(f, \Lambda) = \{M_{\omega}T_{\tau}f : (\tau, \om) \in \Lambda \}$$
where $T_{\tau}f(t)=f(t-\tau)$ and $M_{\omega}f(t)=e^{2\pi i \omega t}f(t)$ are the time and frequency shifts, respectively.

 The following conjecture was formulated by Heil, Ramanathan, and Topiwala in 1996:
 \begin{conj}[\cite{HRT}]\label{conj:HRT}
     Let $f\in L^2(\RR)$ be a nonzero function and $\La\subset\RR^2$ be finite. Then, $\mc{G}(f,\La)$ is linearly independent.  
 \end{conj}
  This statement is known as the \emph{HRT conjecture}. There are several partial results on this conjecture which impose conditions on either the function $f$ or on the set $\Lambda$, or on both of them. Linnell showed that the conjecture holds when $\Lambda$ is a finite subset of a lattice, using the von Neumann subalgebra of $\mathcal{B}(L^2(\mathbb{R}))$ generated by time-frequency shifts on this set \cite{Linnell}. The same result has been proven in the context of the theory of shift invariant spaces, that Bownik developed in \cite{SIspaces}, by Bownik and Speegle in \cite{wavelets} and also by Demeter and Gautam in \cite{schrodinger}. Other important results are based on the geometry of the set. In particular, Demeter and Zaharescu obtained results in \cite{13, 22} when $\Lambda$ has four points contained in two parallel lines. In addition, the results by Bownik and Speegle in \cite{exponentialdecay} and by Benedetto and Bourouihiya in \cite{infbehavior} show that the behavior of the function at infinity plays an important role.

To emphasize the difficulty of proving Conjecture \ref{conj:HRT}, we mention that several weaker versions are still open, for example, if we restrict to $f\in\mc{S}(\RR)$ (the Schwartz class) or $\abs{\La}=4$ with $\La$ not on a lattice. For the $4$-point case \cite{13, 22}, one needs deep quantitative results from Diophantine approximation, which, roughly speaking, indicate how far $\La$ is from a lattice.

The novelty of our approach is the utilization of a different tool from Diophantine approximation, namely the \emph{Lonely Runner Conjecture} \cite{willslrc} and its shifted variant \cite{sLRC}, in particular, a special case proved by the first author \cite{LRC}. With this tool, we improve the results for Conjecture \ref{conj:HRT} for ultimately positive functions in \cite{infbehavior}. 

\begin{defn}
    A function $f:\RR\to\CC$ is called \emph{ultimately positive}, if there is some $t_0\in\RR$ such that $f(t)>0, \forall t\geq t_0$.
\end{defn}

\begin{defn}
    Let $\Om\subset\RR$ be finite. The \emph{affine dimension} of $\Om$ over $\QQ$ is the dimension of the $\QQ$-vector space spanned by $\Om-\om$, where $\om\in\Om$. 
\end{defn}

We remark that the affine dimension does not depend on the choice of $\om\in\Om$ in the definition above. Our first main result is:

\begin{thm}\label{thm:mainonedependence}
Let $f \in L^2(\mathbb{R})$ be an ultimately positive function and $\Lambda=\set{(\tau_k,\om_k)}_{k=0}^N$, such that the affine dimension of $\Omega=\set{\om_0,\dotsc,\om_N}$ over $\QQ$ is at least $N-1$. Then, $\mathcal{G}(f, \Lambda)$ is linearly independent.    
\end{thm}


Benedetto and Bourouihiya proved the above Theorem with the stronger assumption that $\Om$ is linearly independent over $\QQ$, even though they only use the fact that $\Om-\om_0$ is linearly independent over $\QQ$, for some $\om_0\in\Om$.

The second main result of this paper is:

\begin{thm}\label{thm:main4pt}
Let $f\in L^2(\mathbb{R})$ be an ultimately positive function and $\Lambda \subset \mathbb{R}^2$ a set of four points. Then, $\mathcal{G}(f, \Lambda)$ is linearly independent.
\end{thm}

Again, the above Theorem strengthens a result in \cite{infbehavior}; Theorem 5.6 in \cite{infbehavior} proves linear independence of four time-frequency shifts of an ultimately positive function $f$, with the additional hypotheses that both $f(x)$ and $f(-x)$ are ultimately positive and decreasing.

The proof of Theorem \ref{thm:main4pt} is in Section \ref{sec:4pt} and is divided into five cases, depending on whether some points of $\La$ belong to the same horizontal line. The last two cases, tackled in subsections \ref{subsec:4} and \ref{subsec:22}, correspond to configurations where all elements of $\La$ are in the same line or in two horizontal lines (each line having exactly two elements of $\La$), respectively. These cases may be already known in \cite{HRT} and \cite{22}, nevertheless, we prove them here again for the special case where $f$ is ultimately positive. It should also be remarked that these are the cases where some sort of decay condition or integrability is needed for $f$; besides it is known that there are functions $f$ that are positive and periodic and satisfying nontrivial linear dependency relations among certain time-frequency translates thereof \cite{HRT}. For example, for $f(x)=2+\cos 2\pi x$ and $\La=\set{(0,0), (0,-1), (0,1), (a,0), (a,-1), (a,1)}$, $\mc{G}(f,\La)$ is linearly dependent for any $a\in\RR$, showing that HRT conjecture cannot be extended to $L^p(\TT)$.


\section{Preliminaries}

We can apply some transformations to a Gabor system, in order to change the geometry of the set $\Lambda$, without losing the property of linear independence. In particular we have the following propositions which can be found in \cite{HRT}.

\begin{prop}\label{prop:metaplectic}
If $A: \mathbb{R}^2 \rightarrow  \mathbb{R}^2$ is a linear transformation, with $detA=1$, then there exists a unitary transformation $U_{A}:  L^2(\mathbb{R}) \rightarrow  L^2(\mathbb{R})$
$$U_{A}\mathcal{G}(f, \Lambda) = \{U_{A}M_{\omega}T_{\tau}f \}_{(\tau, \om)\in\La}=\{ c(\tau, \om) M_{v}T_{u}(U_{A}f)\}_{(u,v) \in A(\Lambda)}$$
where $c(\tau, \omega)\in \mathbb{T}$, are complex numbers modulus 1.
\end{prop}

\begin{prop}
Let $\mathcal{G}(f, \Lambda)$ be a finite Gabor system and $U_{A}:  L^2(\mathbb{R}) \rightarrow  L^2(\mathbb{R})$ a metaplectic transformation with associated linear transformation $A: \mathbb{R}^2 \rightarrow  \mathbb{R}^2$. Then
$\mathcal{G}(f, \Lambda) $ is linear independent if and only if  $\mathcal{G}(U_{A}f, A(\Lambda))$ is linear independent.
\end{prop}

\begin{prop}\label{prop:halfline}
    Let $\mc{G}(f,\La)$ be a finite Gabor system, with $f\in L^2(\RR)$ supported on a half-line. Then, $\mc{G}(f,\La)$ is linearly independent.
\end{prop}

However, when focusing on a certain class of functions, one needs to ensure whether $U_Af$ still belongs to the same class. In our case, the class of ultimately positive functions is \emph{shift-invariant} (though not exactly in the sense mentioned in \cite{SIspaces}), namely $f$ is ultimately positive if and only if every $T_\tau f$ is ultimately positive $\forall \tau\in\RR$.

For this purpose, we will need a weaker but simpler result.

\begin{prop}\label{prop:translatingLambda}
    Let $f$ be a nonzero measurable function, $\La\subset\RR^2$ a finite set and $\la=(\tau,\om)\in\RR^2$. Then, the Gabor system $\mc{G}(f,\la+\La)$ is linearly dependent if and only if $\mc{G}(T_\tau f,\La)$ is linearly dependent.
\end{prop}

\begin{proof}
    Let $\La=\set{(\tau_j, \om_j):1\leq j\leq k)}$. Suppose that $c_1,\dotsc,c_k\in\CC$ are such that
    \begin{equation}\label{eq:lindeptranslation1}
    \sum_{j=1}^k c_jM_{\om+\om_j}T_{\tau+\tau_j}f=0.    
    \end{equation}
    This holds if and only if for almost all $t\in\RR$ it holds
    \[0=\sum_{j=1}^k c_jM_\om M_{\om_j}T_{\tau_j}(T_\tau f)(t)=\sum_{j=1}^k c_je^{2\pi i\om t}M_{\om_j}T_{\tau_j}T_\tau f(t).\]
    So, $\mc{G}(f,\la+\La)$ is linearly dependent if and only if there is a selection of coefficients $c_1,\dotsc,c_k$ satisfying \eqref{eq:lindeptranslation1} such that $c_j\neq0$ for some $j$. The latter is obviously equivalent to $\mc{G}(T_\tau f,\La)$ being linearly dependent.
\end{proof}

The following is an immediate consequence of Proposition \ref{prop:translatingLambda}.

\begin{cor}\label{cor:wlog0inLambda}
    Let $\mc{C}$ be a class of functions defined on $\RR$ that is shift-invariant, in the sense that $f\in\mc{C}$ if and only if $T_\tau f\in\mc{C}, \forall\tau\in\RR$. Let $\La\subset\RR^2$ be finite and $\la\in\RR^2$. Then, $\mc{G}(f,\La)$ is linearly independent for every $f\in\mc{C}$ if and only if $\mc{G}(f,\la+\La)$ is linearly independent for every $f\in\mc{C}$.
\end{cor}

In other words, to check if $\mc{G}(f,\La)$ is linearly independent for an ultimately positive function $f$, we may assume without loss of generality that $(0,0)\in\La$.

\section{Simultaneous Approximation on the Unit Torus}

Benedetto and Bourouihiya  proved that the HRT conjecture holds for $\mathcal{G}(f, \Lambda)$ when $f$ is an ultimately positive function and $\Lambda=\set{(\tau_k,\om_k)}_{k=0}^N$ 
have the property that $\Omega=\set{\om_0,\dotsc,\om_N}$ is linearly independent over $\mathbb{Q}$ \cite{infbehavior}. The proof is based on an argument of simultaneous approximation on the set of frequencies $\Omega$, using Kronecker's theorem. 
 The key observation is that we do not have to restrict to natural numbers to achieve simultaneous approximation. The next lemma, which was the starting point for Y. Meyer to introduce the important concept of \textit{harmonious sets} \cite{Meyer72hamronious}, is the main ingredient that allow us to modify the proof of Benedetto and Bourouihiya \cite{infbehavior}.

\begin{lemma}\label{lem:simapprox}
Let $\Lambda=\set{\lambda_1,\dotsc,\lambda_n}$ an arbitrary subset of real numbers and $x_1,\dotsc,x_n$ also an arbitrary sequence of real numbers. Then the following are equivalent:
\begin{enumerate}
    \item For each $\epsilon>0$ there is a real number $t$ such that $$\| t \lambda_j-x_j \|_{\mathbb{R}/\mathbb{Z}} \leq \epsilon, \ \ \   1 \leq j \leq n.$$
    \item For each sequence of rational integers $p_1,\dotsc,p_n$, we have the implication 
\begin{eqnarray}
p_1 \lambda_1+\dotsb+p_n \lambda_n=0 \implies p_1 x_1+\dotsb+p_n x_n =0 (mod1)    
\end{eqnarray}
    
\end{enumerate}
\end{lemma}

Before we prove the lemma we notice that condition \textit{(2)}, can be restated as follows:
Let $H=\langle \Lambda \rangle$ be the subgroup of $ \mathbb{R}$ generated by $\Lambda=\{ \lambda_1,\dotsc,\lambda_n \}$ then there is a homomorphism $\chi:H \rightarrow \mathbb{R}/\mathbb{Z}$ such that $$\chi(\lambda_{j})=x_{j} (mod1)$$ for each $ 1 \leq j \leq n$.

\begin{proof} $(1) \implies (2)$ Let $\epsilon>0$. If $p_1 \lambda_1+...+p_n \lambda_n=0$ then for each $t \in \mathbb{R}$,\\ $p_1 t\lambda_1+\dotsb+p_n t\lambda_n=0$. Thus
$$\| p_1 x_1+\dotsb+p_n x_n\|_{\mathbb{R}/\mathbb{Z}}=\| p_1 x_1+...+p_n x_n-(p_1 t\lambda_1+\dotsb+p_n t\lambda_n)\|_{\mathbb{R}/\mathbb{Z}} \leq$$
 $$\| p_1 (x_1-t\lambda_1) \|_{\mathbb{R}/\mathbb{Z}}+\dotsb+ \| p_n (x_n-t\lambda_n) \|_{\mathbb{R}/\mathbb{Z}}\leq C \epsilon$$
where $C=max_j |p_j|$ and the last inequality holds due to $(1)$.\\
$(2) \implies (1)$ By structure theorem for submodules of free modules on principal rings, we have that
$$H=\langle \Lambda \rangle = \mathbb{Z}\omega_1 \oplus ...\oplus \mathbb{Z}\omega_m$$
where $m \leq n$ and $\omega_1,\dotsc,\omega_m$ are linear independent over $\mathbb{Q}$. The observation after the lemma give us a homomorphism $\chi:H \rightarrow \mathbb{R}/\mathbb{Z}$ such that $\chi(\lambda_{j})=x_{j} (mod1)$ for each $ 1 \leq j \leq n$. For each $\epsilon>0$, from Kronecker's theorem, there exists a real number $t$ such that $$\| t \omega_j-\chi(\omega_j) \|_{\mathbb{R}/\mathbb{Z}} \leq \epsilon, \ \ \   1 \leq j \leq m$$
Since $\lambda_j \in H$ for each $1 \leq j \leq n$:
$$\| t \lambda_j-x_j \|_{\mathbb{R}/\mathbb{Z}}=\| t \lambda_j-\chi(\lambda_j) \|_{\mathbb{R}/\mathbb{Z}}=\norm*{t \sum_{j=1}^{m} k_j \omega_j-\chi(\sum_{j=1}^{m} k_j \omega_j)}_{\mathbb{R}/\mathbb{Z}}\leq $$

\[\norm*{\sum_{j=1}^{m} k_j t \omega_j- \sum_{j=1}^{m} k_j \chi (\omega_j)}_{\mathbb{R}/\mathbb{Z}} \leq \sum_{j=1}^{m} \abs{k_j} \norm{t \omega_j-\chi(\omega_j)}_{\mathbb{R}/\mathbb{Z}} \leq C \epsilon. \qedhere\]
\end{proof}

The Lemma \ref{lem:simapprox} gives a necessary and sufficient condition to simultaneously approximate a finite sequence of reals from a given set.  
\begin{defn}
 For a set $\Lambda=\{ \lambda_1,...,\lambda_n \}$ of real numbers we say that $x_1,...,x_n$ is a \textit{good sequence} for $\Lambda$ if it satisfies (4.1), otherwise we say that it is a \emph{bad sequence} for $\Lambda$. We can use the same definition for a sequence on the unit torus via the obvious isomorphism. 
\end{defn}

\begin{proof}[Proof of Theorem \ref{thm:mainonedependence}]
Suppose that $\om_0\leq\om_1\leq\dotsb\leq\om_N$ and let $\Omega=\set{\om_0,\dotsc,\om_N}$. If $\Om-\om_0$ is linearly independent over $\mathbb{Q}$ we have that the Gabor system $\mathcal{G}(f, \Lambda)$ is linearly independent from the proof of Theorem 5.6 in \cite{infbehavior}. So, we assume that there is exactly one linear dependence on the set of frequencies, say 
 \begin{eqnarray}
 p_1 \omega_1+\dotsb+p_N \omega_N=0, \ \ \ \ p_1,\dotsc,p_N \in \mathbb{Z}, \ \ \ \ \gcd(p_1,\dotsc,p_N)=1    
 \end{eqnarray}
The linear dependence of $ \mathcal{G}(f, \Lambda)$ implies that there are constants $c_1, c_2,\dotsc, c_N \in \mathbb{C} \setminus \{ 0 \}$ such that
\begin{equation}\label{eq:lindepmany}
f(t)= \sum_{k=1}^{N}c_k e^{2\pi i \omega_k t}f(t-\tau_k)   \  \  \  \  \  \  \ a.e.
\end{equation}
Define $\theta_k \in \mathbb{R}, 1\leq k\leq N$ to satisfy $e^{2\pi i \theta_k}=\frac{|c_k|}{c_k}i$ for $k=1,\dotsc,N$. Let 
$$p_1 \theta_1+\dotsb+p_N \theta_N=\alpha.$$
Our goal is to perturb $\theta_k$ as little as necessary to $\theta_k+\vphi_k$, such that $\theta_k+\vphi_k$ is a good sequence for $\om_1,\dotsc,\om_N$. In particular, we want $\abs{\vphi_k}<\frac{1}{4}$, so that $\mathrm{Im}(c_ke^{2\pi i(\theta_k+\vphi_k)})>0$, for all $k=1,\dotsc,N$.
Since
\[\sum_{k=1}^Np_k(\theta_k+\vphi_k)\in\ZZ,\]
we must have 
\[\sum_{k=1}^Np_k\vphi_k\equiv-\al\bmod1.\]
The smallest possible absolute value of an element in $-\al\bmod1$ is at most $1/2$. On the other hand, all the possible values of the sum $\sum_{k=1}^Np_k\vphi_k$ form the interval
\[\bra{-\frac{1}{4}\sum_{k=1}^N\abs{p_k},\frac{1}{4}\sum_{k=1}^N\abs{p_k}}.\]
Hence, if $\sum_{k=1}^N\abs{p_k}\geq3$, the existence of such $\vphi_k$, $k=1,\dotsc,N$ is guaranteed.

Select $\eps>0$ such that $\abs{\vphi_k}+\eps<\frac{1}{4}$, and consider the set of all $t\in\RR$ such that
\[\norm{t\om_k-(\theta_k+\vphi_k)}_{\RR/\ZZ}\leq\eps, \;\;\; k=1,\dotsc,N,\]
which is nonempty by Lemma \ref{lem:simapprox} and unbounded above. Therefore, the set
\[T=\set{t\in\RR:\mathrm{Im}(c_ke^{2\pi i\om_kt})>0}\]
satisfies $m(T\cap[t_0,\infty))>0$, $\forall t_0\in\RR$. Now, let $t_0\in\RR$ be such that $f(t-\tau_k)>0, \forall t\geq t_0, \forall k=0,1,\dotsc,N$. Taking imaginary parts in \eqref{eq:lindepmany}, we get
\[0=\sum_{k=1}^N\mathrm{Im}(c_ke^{2\pi i\om_kt})f(t-\tau_k),\]
for almost all $t\geq t_0$. But this is clearly a contradiction if $t\in T$.

The only cases that are not covered by the above argument are when $\sum_{k=1}^N\abs{p_k}=1$ or $2$. In the first case, this simply means that one more frequency is zero, say $\om_1$, but all the rest are linearly independent over $\QQ$, whereas in the second one, it means that two frequencies are equal in absolute value, hence equal, since they are all nonnegative, say $\abs{\om_j}=\abs{\om_{j+1}}$, therfore $\om_1,\om_2,\dotsc,\om_j,\om_{j+2},\dotsc,\om_N$ must be linearly independent over $\QQ$.

Applying Corollary \ref{cor:wlog0inLambda}, both cases above may be reduced to $\om_0=\tau_0=\om_1=0$, $\om_2,\dotsc,\om_N$ linearly independent over $\QQ$ (here, not all $\om_k$ are necessarily nonnegative, but this won't be a problem). We rewrite the linear dependence relation as
\begin{equation}\label{eq:lindepmany2}
f(t)+c_1f(t-\tau_1)= \sum_{k=2}^{N}c_k e^{2\pi i \omega_k t}f(t-\tau_k)   \  \  \  \  \  \  \ a.e.
\end{equation}
We may assume without loss of generality that $\mathrm{Im}(c_1)\leq0$. By Kronecker's Theorem, the set $T$ for which 
\[\norm{\om_kt-\theta_k}_{\RR/\ZZ}\leq\eps, \;\;\; 2\leq k\leq N,\]
satisfies $m(T\cap[t_0,\infty))>0$, $\forall t_0\in\RR$, where $\eps>0$; here, we take $\eps<\frac{1}{4}$, so that
\[\mathrm{Im}(c_ke^{2\pi i\om_kt})>0, \;\;\;\forall t\in T,\;\; 2\leq k\leq N.\]
Taking imaginary parts in \eqref{eq:lindepmany2} we get
\[\mathrm{Im}c_1\cdot f(t-\tau_1)=\sum_{k=2}^N\mathrm{Im}(c_ke^{2\pi i\om_kt})f(t-\tau_k),\]
for almost all $t\geq t_0$, where $t_0$ is such that $f(t-\tau_k)>0$, $\forall t\geq t_0$, $0\leq k\leq N$. But this establishes a contradiction, since for every $t$ in the set of positive measure $T\cap[t_0,\infty)$ it holds
\[\mathrm{Im}c_1\cdot f(t-\tau_1)\leq0<\sum_{k=2}^N\mathrm{Im}(c_ke^{2\pi i\om_kt})f(t-\tau_k).\]
This completes the proof.
\end{proof} 

\section{The Lonely Runner Conjecture and its shifted variant}

The \emph{Lonely Runner Conjecture} was initially stated by Wills \cite{willslrc} in the setting of Diophantine approximation: it simply states that if we have $n$ runners with pairwise distinct positive velocities $v_1,\dotsc,v_n$, all starting from the same point on a circular track of length 1, then at some point in time the distance (arc length) of every runner from the start is at least $\frac{1}{n+1}$. There has been some recent work on this problem; we refer the reader to \cite{7lrc, LRC, MSS, tao}.

The \emph{shifted} version of the Lonely Runner Conjecture has a similar statement:
\begin{conj}[Lonely Runners with individual starting points \cite{sLRC}]\label{conj:sLRC}
 Given pairwise distinct positive velocities $v_1,\dotsc,v_n \in \mathbb{R}$
and arbitrary starting points $s_1,\dotsc,s_n \in \mathbb{R}$ 
there is a real number $t$ such that for all $1 \leq j \leq n$ the distance of $s_j+tv_j$ to the nearest integer is at least $\frac{1}{n+1}$.
\end{conj}

The case $n=1$ is trivial and $n=2$ is very easy \cite{sLRC}. In \cite{LRC} the case of three runners was proved for Conjecture \ref{conj:sLRC} and the Conjecture has been further explored in \cite{MSS}; moreover, the extremal cases have been described completely.

\begin{thm}[\cite{LRC}]\label{thm:3runnersslRC}
Consider three runners with pairwise distinct constant velocities $0<v_1<v_2<v_3$, who start running on a circular track of length 1, with not necessarily identical starting positions. A stationary spectator watches the runners from a fixed position along the track. Then, there exists a time at which all the runners have distance at least 1/4 from the spectator.

Moreover, if $v_1:v_2:v_3$ is not proportional to $1:2:3$, then $1/4$ above may be replaced with $1/4+\eps$, for sufficiently small $\eps>0$. 
\end{thm}

For the proof of Theorem \ref{thm:main4pt}, it is crucial to find a set $T$, such that $m(T\cap[\al,\infty))>0$, $\forall\al\in\RR$, such that at every moment $t\in T$, the three runners have distance at least 1/4 from the spectator.

Theorem \ref{thm:3runnersslRC} states that, for distinct positive integer velocities, this always happens except for the extremal case where $v_1:v_2:v_3$ is proportional to $1:2:3$. In this case, a weaker result suffices for our purposes.

\begin{prop}\label{prop:extreme}
    Suppose that three runners have velocities $v_j=j$, $j=1,2,3$, and running on the circular track $S^1\subset\CC$ with arbitrary starting points. Consider three spectators, located at $1$, $i$ and $-i$. Denote by $T_k$ the set of moments for which the three runners have distance greater than 1/4 from the spectator located at $k$. Then, for at least one $k\in\set{1,i,-i}$, it holds $m(T_k\cap[\al,\infty))>0$, $\forall \al\in\RR$.
\end{prop}

\begin{proof}
    The sets $T_k$ are all periodic with period $1$, so the condition $m(T_k\cap[\al,\infty))>0$, $\forall \al\in\RR$, holds if and only if $m(T_k)>0$. It suffices to prove that at least one of the sets $T_k$, $k=1,i,-i$, contains an interval.

    On the other hand, we know by Theorem \ref{thm:3runnersslRC} that all $T_k$ are nonempty. Suppose that $m(T_1)=0$; then, for every $t\in T_1$, at least one of the runners must be at the position $i$ and another one at $-i$. Indeed, if $t_0\in T_1$ and no runner is at $i$ (or $-i$), then there is an interval containing $t_0$ which is a subset of $T_1$.

    Let $v_j(t)=jt+s_j$ the position of the $j$th runner. Suppose that at time $t_0$, the second runner is at $i$, i.e. $e^{2\pi iv_2(t_0)}=i$. Then, either $e^{2\pi iv_j(t_0)}=-i$, for $j=1$ or $3$. Therefore, at $t'=t_0+\frac{1}{2}$, both the second and the $j$th runner are at $i$, hence the remaining runner must be at $-i$. This argument shows in general that if $m(T_k)=0$, then for any $t\in T_k$, the first and the third runner occupy the ``extreme'' positions, which are exactly 1/4 distance from the spectator at $k$.

    So, continuing our assumption that $m(T_1)=0$, let $t_0\in T_1$ such that $e^{2\pi iv_1(t_0)}=i$ and $e^{2\pi iv_3(t_0)}=-i$ (if it was $e^{2\pi iv_1(t_0)}=-i$ and $e^{2\pi iv_3(t_0)}=i$, we would consider the moment $t_0+\frac{1}{2}$). Since the second runner also has distance at least $1/4$ from the spectator at $1$, $e^{2\pi iv_2(t_0)}$ would be in the left semicircle, i.~e. $\cos(2\pi v_2(t_0))\leq0$. At the moment $t_1=t_0+\frac{1}{4}$ it holds 
    \[e^{2\pi iv_1(t_1)}=e^{2\pi iv_3(t_1)}=-1.\]
    If $e^{2\pi iv_2(t_1)}\neq1$, then either $m(T_{-i})>0$ or $m(T_i)>0$, depending on whether $\sin(2\pi v_2(t_1))\geq0$ or $\sin(2\pi v_2(t_1))\leq0$. If $e^{2\pi iv_2(t_1)}=1$, then at the moment $t_2=t_1+\frac{1}{2}$ it holds
    \[e^{2\pi iv_1(t_2)}=e^{2\pi iv_2(t_2)}=e^{2\pi iv_3(t_2)}=1,\]
    therefore, $[t_2,t_2+\frac{1}{6}]\in T_{-i}$, completing the proof.
\end{proof}


\section{Proof of Theorem \ref{thm:main4pt}}\label{sec:4pt}

Suppose that $ \mathcal{G}(f, \Lambda)$ is a linearly dependent Gabor system, where $\Lambda= \{(\tau_k, \om_k) \}_{k=0} ^{3}$ and without loss of generality, by Corollary \ref{cor:wlog0inLambda}, we assume that $(\tau_0, \om_0)=(0,0)$ and $\omega_0 \leq \omega_1 \leq \omega_2 \leq \omega_3$. Lastly, denote $\Om=\set{\om_0,\om_1,\om_2,\om_3}$.

\subsection{\textbf{Case 1:}  \texorpdfstring{$\omega_0 < \omega_1 < \omega_2 < \omega_3$}{}}

The linear dependence implies that there are constants $c_1, c_2, c_3 \in \mathbb{C} \setminus \{ 0 \}$ such that
\begin{equation}\label{eq:lineardependence4points}
f(t)= \sum_{k=1}^{3}c_k e^{2\pi i \omega_k t}f(t-\tau_k),
\end{equation}
for almost all $t\in\RR$. If the affine dimension of $\Om$ is at least $2$, then we have a contradiction by Theorem \ref{thm:mainonedependence}. So, we may assume that the affine dimension of $\Om$ is $1$, which means that $\om_1,\om_2, \om_3$ are proportional to positive integers.

We denote by $\theta_k$ the arguments of the coefficients $c_k$, i.~e. $c_k=\abs{c_k}e^{2\pi i\theta_k}$.
Consider three runners with velocities $\om_1, \om_2, \om_3$,
with starting points on the unit circle $S^1$ of $\CC$, the numbers $e^{2\pi i \theta_k} = c_k / \abs{c_k}$, for $k=1,2,3$. Suppose that the spectator lies at the point $z=1$. The fact that from any starting position, the three runners eventually have (arc length) distance at least $\frac{1}{4}$ from $z=1$, by virtue of Theorem \ref{thm:3runnersslRC}, this means that we can find a set $T$ such that $T\cap[\al,\infty)\neq\vn$ for all $\al\in\RR$ and
\begin{equation}\label{eq:negativerealparts}
    \mathrm{Re}(e^{2\pi i( \theta_k + \omega_k t)})=\cos(2\pi( \omega_kt+\theta_k)) \leq 0, \ \ \forall t \in T, \ \ \forall k=1,2,3
\end{equation}

Now, let $\al\in\RR$ be such that $f(t-\tau_k)>0, \forall k\in\set{0,1,2,3}, \forall t\geq\al$. Then, taking real parts in \eqref{eq:lineardependence4points} we get
\begin{equation}\label{eq:tn}
f(t)= \sum_{k=1}^{3}\abs{c_k} \cos(2\pi( \omega_kt+\theta_k))f(t-\tau_k),
\end{equation}
for almost every $t\geq\al$. If $\om_1:\om_2:\om_3$ is not proportional to $1:2:3$, then $m(T\cap[\al,\infty))>0$, where $m$ is the usual Lebesgue measure, establishing a contradiction by Theorem \ref{thm:3runnersslRC}, since for almost all $t\in T\cap[\al,\infty)$ we have
\begin{equation}\label{eq:realpartsineq}
f(t)>0\geq\sum_{k=1}^{3}\abs{c_k} \cos(2\pi( \omega_kt+\theta_k))f(t-\tau_k).    
\end{equation}
It remains to consider the case where $\om_1:\om_2:\om_3$ is proportional to $1:2:3$. For this purpose, we invoke Proposition \ref{prop:extreme} and use the same notation for the sets $T_1$, $T_i$, $T_{-i}$.

If $m(T_1\cap[\al,\infty))>0$, $\forall \al\in\RR$, we establish a contradiction by taking real parts in \eqref{eq:lineardependence4points} and picking $\al$ such that $f(t-\tau_k)>0$, $\forall t\geq\al$, $k=0,1,2,3$. Again we get \eqref{eq:negativerealparts} for $T=T_1$ and \eqref{eq:tn} for almost all $t\geq\al$, contradicting \eqref{eq:realpartsineq}, which holds for almost all $t\in T_1\cap[\al,\infty)$.

If $m(T_i\cap[\al,\infty))>0$, $\forall \al\in\RR$, we establish a contradiction by taking imaginary parts in \eqref{eq:lineardependence4points} and picking $\al$ such that $f(t-\tau_k)>0$, $\forall t\geq\al$, $k=0,1,2,3$. We would have 
\begin{equation}\label{eq:negativeimaginaryparts}
    \mathrm{Im}(e^{2\pi i( \theta_k + \omega_k t)})=\sin(2\pi( \omega_kt+\theta_k)) < 0, \ \ \forall t \in T_i, \ \ \forall k=1,2,3.
\end{equation}
We would also have 
\begin{equation}\label{eq:negimcontr}
0= \sum_{k=1}^{3}\abs{c_k} \sin(2\pi( \omega_kt+\theta_k))f(t-\tau_k),
\end{equation}
for almost every $t\geq \al$. However, for every $t\in T_i\cap[\al,\infty)$ we have 
\begin{equation*}
0>\sum_{k=1}^{3}\abs{c_k} \sin(2\pi( \omega_kt+\theta_k))f(t-\tau_k),   
\end{equation*}
a contradiction.

The case $m(T_{-i}\cap[\al,\infty))>0$, $\forall\al\in\RR$ is similar to the previous one; the only difference is that 
\begin{equation}\label{eq:positiveimaginaryparts}
    \mathrm{Im}(e^{2\pi i( \theta_k + \omega_k t)})=\sin(2\pi( \omega_kt+\theta_k)) > 0, \ \ \forall t \in T_{-i}, \ \ \forall k=1,2,3.
\end{equation}

\subsection{\textbf{Case 2:}  \texorpdfstring{$\omega_0 = \omega_1 < \omega_2 < \omega_3$}{}}

The linear dependence relation becomes
\begin{equation}\label{eq:precase2}
f(t)+c_1f(t-\tau_1)=\abs{c_2}e^{2\pi i (\omega_2 t+ \theta_2)}f(t-\tau_2)+\abs{c_3}e^{2\pi i ( \omega _3 t + \theta_3)}f(t-\tau_3).  
\end{equation}
We will actually show something stronger in this case:
\begin{claim}\label{claim:halfline}
    There is no function $f:\RR\to\CC$ satisfying \eqref{eq:precase2} and $f(t-\tau_k)>0$, for all $k\in\set{0,1,2,3}$ and for almost all $t\in[\al,\infty)$.
\end{claim}
\begin{proof}[Proof of Claim]
Assume that the contrary holds for some function $f$;
taking imaginary parts on both sides of \eqref{eq:precase2}, we get
\begin{equation}\label{eq:case2}
\mathrm{Im}c_1\cdot f(t-\tau_1)=\abs{c_2}\sin (2\pi( \omega_2 t+ \theta_2))f(t-\tau_2)+  \abs{c_3}\sin (2 \pi ( \omega _3 t + \theta_3)) f (t-\tau_3),
\end{equation}
for almost all $t\in[\al,\infty)$.
Without loss of generality we assume that $\mathrm{Im}c_1\geq0$ and we can choose $t_0\geq\al$, such that 
$$\sin(2 \pi (\omega_k t_0+ \theta_k)) \leq -\frac{1}{2} , \;\;\; k=2,3,$$ 
by considering two runners on $S^1$ with velocities $\om_2, \om_3$, starting points at $e^{2\pi i(\om_k\al+\theta_k)}$, $k=2,3$, and the spectator at $z=i$ (we apply Conjecture \ref{conj:sLRC} for two runners). Therefore, the weaker inequalities $\sin(2 \pi (\omega_k t+ \theta_k))<0, k=2,3,$ hold for a subset of positive measure in $[t_0,\infty)$; this contradicts \eqref{eq:case2}. 
\end{proof}

This argument also applies in the cases where $\omega_0 < \omega_1 = \omega_2 < \omega_3$ or $\omega_0 < \omega_1 < \omega_2 = \omega_3$.

\subsection{\textbf{Case 3:} \texorpdfstring{$\omega_0 = \omega_1 = \omega_2 < \omega_3$}{}}

In this case, we have
$$f(t)+c_1f(t- \tau_1)+c_2f(t- \tau_2)=|c_3|e^{2\pi i ( \omega _3 t + \theta_3)}f(t-\tau_3).$$
Again, taking imaginary parts on both sides, we obtain
$$\mathrm{Im}c_1\cdot f(t-\tau_1)+ \mathrm{Im}c_2\cdot f(t-\tau_2)=\abs{c_3}\sin (2 \pi ( \omega _3 t + \theta_3)) f(t-\tau_3),$$
for almost all $t$ in a half line, say $[\al,\infty)$, such that $f(t-\tau_k)>0$ for all $k\in\set{0,1,2,3}$ almost all $t\geq\al$.
If $c_1'=\mathrm{Im}c_1, c_2'= \mathrm{Im}c_2$, then
\begin{equation}\label{eq:case3}
c_1'f(t-\tau_1)+ c_2'f(t-\tau_2)=| c_3 | \frac{e^{2\pi i ( \omega _3 t + \theta_3)}-e^{2\pi i (-\omega _3 t - \theta_3)}}{2i}f(t-\tau_3),    
\end{equation}
for almost all $t\in[\al,\infty)$, establishing a contradiction using Claim \ref{claim:halfline} above with $f$ and the set of time-frequency translates $\Lambda'=\{ (0, \tau_1), (0, \tau_2) , (\omega_3, \tau_3), (-\omega_3, \tau_3)\}$.

\smallskip
\noindent
The remaining cases, when all the frequencies are the same $\omega_0 = \omega_1 = \omega_2 = \omega_3$ so that $\Lambda$ forms a line \cite{HRT} or a (2,2) configuration \cite{22} $\omega_0 = \omega_1 < \omega_2 = \omega_3$ are known results for all nonzero $f\in L^2(\RR)$. However, we will include them here for the special case where $f$ is ultimately positive, for completion. It should be noted that these are the first cases where we use some decay conditions for $f$, implied by square-integrability, besides ultimate positivity.

\subsection{\textbf{Case 4:} \texorpdfstring{$\om_0=\om_1=\om_2=\om_3$}{}}\label{subsec:4}

It is known that the Fourier transform rotates the time-frequency plane by $\pi/2$; in other words, it is the unitary transformation $U_A$ described in Proposition \ref{prop:metaplectic} for $A(\tau,\om)=(-\om,\tau)$. Therefore, $\mc{G}(f,\La)$ is linearly dependent if and only if $\mc{G}(\hat{f},A(\La))$ is. The latter is equivalent to the existence of $c_0,\dotsc,c_3\in\CC$, not all zero, such that
\[\hat{f}(\om+\om_0)\sum_{j=0}^3 c_je^{2\pi i\tau_j\om}=0,\]
for almost all $\om\in\RR$, a contradiction, since the zero set of a (generalized) trigonometric polynomial has measure zero, and $\hat{f}$ cannot be zero on a set of full measure.

\subsection{\textbf{Case 5:} \texorpdfstring{$\omega_0 = \omega_1 < \omega_2 = \omega_3$}{}}\label{subsec:22}

Without loss of generality, we assume that $\tau_1>0$ and $\tau_3>\tau_2$. We rewrite the dependence relation as
\begin{equation}\label{eq:precase5}
f(t)+c_1f(t-\tau_1)=\abs{c_2}e^{2\pi i (\omega_2 t+ \theta_2)}f(t-\tau_2)+\abs{c_3}e^{2\pi i ( \omega _2 t + \theta_3)}f(t-\tau_3).  
\end{equation}
This case is similar to Case 2; we assume without loss of generality that $\mathrm{Im}c_1\geq0$ and then use the same method as in the proof of Claim \ref{claim:halfline} to arrive to the same conclusion as in Case 2, as long as the inequalities $\sin(2 \pi (\omega_2 t+ \theta_k))<0, k=2,3,$ hold for a subset of positive measure in $[t_0,\infty)$; the only way this does not happen, is if $\theta_2-\theta_3=\frac{2\ell+1}{2}$, for $\ell\in\ZZ$. Then,
\[\sin(2 \pi (\omega_2 t+ \theta_2))=-\sin(2 \pi (\omega_2 t+ \theta_3)).\]
If $\mathrm{Im}c_1>0,$ We divide \eqref{eq:precase5} by $e^{2\pi i(\om_2t+\theta_2)}$, obtaining
\begin{equation}\label{eq:case5imc1pos}
    e^{-2\pi i(\om_2t+\theta_2)}f(t)+c_1e^{-2\pi i(\om_2t+\theta_2)}f(t-\tau_1)=\abs{c_2}f(t-\tau_2)-\abs{c_3}f(t-\tau_3).
\end{equation}
Taking imaginary parts on both sides, we get
\[-\sin(2\pi(\om_2t+\theta_2))f(t)-\abs{c_1}\sin(2\pi(\om_2t+\theta_2-\theta_1))f(t-\tau_1)=0,\]
where $e^{2\pi i\theta_1}=c_1/\abs{c_1}$. Since $\mathrm{Im}c_1>0$, we may have $\abs{\theta_1}<\frac{1}{2}$, so that $\theta_2-(\theta_2-\theta_1)$ is not a half integer, therefore, there is a set of positive measure in $[t_0,\infty)$ such that both $\sin(2\pi(\om_2t+\theta_2))$ and $\sin(2\pi(\om_2t+\theta_2-\theta_1))$ are negative, contradicting \eqref{eq:case5imc1pos}.

We again establish a contradiction if $\mathrm{Im}c_1=0$; taking imaginary parts in \eqref{eq:precase5}, we obtain
\[\abs{c_2}f(t-\tau_2)-\abs{c_3}f(t-\tau_3)=0,\]
for almost all $t\geq t_0$, or equivalently,
\begin{equation}\label{eq:case5funceq1}
    f(t)=Cf(t-\tau), \;\;\;\text{ for almost all }t\geq t_0-\tau_2,
\end{equation}
where $\tau=\tau_3-\tau_2$ and $C=\frac{\abs{c_3}}{\abs{c_2}}$. Hence, \eqref{eq:precase5} becomes $f(t)+c_1f(t-\tau_1)=0$, for almost all $t\geq t_0$, or equivalently,
\begin{equation}\label{eq:case5funceq2}
    f(t)=cf(t-\tau_1), \;\;\;\text{ for almost all }t\geq t_0,
\end{equation}
where $c=-c_1$. Since $\tau,\tau_1>0$ and $f$ is ultimately positive, we must have $c, C>0$; moreover, since $f\in L^2(\RR)$, we must definitely have $c, C<1$. We distingusih two subcases:

\subsubsection{\texorpdfstring{$\boxed{\tau/\tau_1\in\QQ}$}{}}
 Suppose that $\tau/\tau_1=m/n$, where $m,n\in\NN$ and $\gcd(m,n)=1$. Put $\tau'=\frac{\tau}{m}=\frac{\tau_1}{n}$. Also, let $X$ be the zero measure subset of $[t_0,\infty)$ where either \eqref{eq:case5funceq1} or \eqref{eq:case5funceq2} fails. Then, the set $\tilde{X}=X+\ZZ\tau+\ZZ\tau_1$ is also of measure zero.  Let $t\geq t_0$, such that $t\notin\tilde{X}$. Then, the following equalities hold:
\[f(t)=c^{-m}f(t+m\tau_1)=c^{-m}f(t+n\tau)=c^{-m}C^{n}f(t),\]
hence $C^n=c^m$, or equivalently, $C^{\tau_1}=c^\tau$. Put $c'=\sqrt[m]{C}=\sqrt[n]{c}$ and consider $a,b\in\ZZ$ such that $am+bn=1$; we may take $a>0>b$. It holds $a\tau+b\tau_1=(a\frac{m}{n}+b)\tau_1=\frac{\tau_1}{n}=\tau'$. Then, the following equalities hold, for $t\geq t_0$, $t\notin\tilde{X}$:
\[f(t-\tau')=f(t-a\tau-b\tau_1)=C^{-a}f(t-b\tau_1)=c^{-b}C^{-a}f(t)=c'^{-am-bn}f(t)=\frac{1}{c'}f(t),\]
or equivalently, $f(t)=c'f(t-\tau')$, for almost all $t\geq t_0$. Hence, the function $g(t)=f(t)-c'f(t-\tau')$ is supported on a left half-line of $\RR$. Furthermore, $g\in L^2(\RR)$ and \eqref{eq:precase5} may be transformed to a linear dependence relation of time-frequency translates of $g$: indeed, as
\[f(t)+c_1f(t-\tau_1)=f(t)-c'^nf(t-n\tau')=\sum_{k=0}^{n-1} c'^kg(t-k\tau')=\sum_{k=0}^{n-1}c'^kT_{k\tau'}g(t),\]
and
\begin{align*}
\abs{c_2}e^{2\pi i(\om_2t+\theta_2)}(f(t-\tau_2)-Cf(t-\tau_3))&=\abs{c_2}e^{2\pi i(\om_2t+\theta_2)}(f(t-\tau_2)-c'^mf(t-\tau_2-m\tau'))\\
&=\abs{c_2}e^{2\pi i(\om_2t+\theta_2)}\sum_{j=0}^{m-1}c'^jg(t-\tau_2-j\tau')\\
&=\sum_{j=0}^{m-1}c_2c'^jM_{\om_2}T_{\tau_2+j\tau'}g(t),
\end{align*}
yielding
\[\sum_{k=0}^{n-1}c'^kT_{k\tau'}g(t)-\sum_{j=0}^{m-1}c_2c'^jM_{\om_2}T_{\tau_2+j\tau'}g(t)=0,\]
contradiction by Proposition \ref{prop:halfline}.

\subsubsection{\texorpdfstring{$\boxed{\tau/\tau_1\notin\QQ}$}{}}
We still have $C^{\tau_1}=c^\tau$, but the proof is different. Suppose on the contrary, that $C^{\tau_1}\neq c^\tau$, without loss of generality $C^{\tau_1}<c^\tau$. Since $0<c, C<1$, there is some $\de>0$ such that $c^{\tau+\de}=C^{\tau_1}$. Next, take $\eps>0$ such that $\eps<\de\log_Cc$. Since $\tau/\tau_1\notin\QQ$ and $\tau,\tau_1>0$, there are $m,n\in\NN$ such that
\[0<m\tau_1-n\tau<\eps.\]
By \eqref{eq:case5funceq1} and \eqref{eq:case5funceq2} we get
\begin{equation}\label{eq:case5funceq3}
    f(t+m\tau_1-n\tau)=C^{-n}c^mf(t), \;\;\;\text{ for almost all }t\geq t_0.
\end{equation}
We will establish a contradiction by showing that $C^{-n}c^m>1$; such a functional equation cannot be satisfied by an ultimately positive function in $L^2(\RR)$. Indeed,
\begin{align*}
    c^m=(c^\tau)^{m/\tau}=(c^{\tau+\de})^{m/\tau}c^{-\de\frac{m}{\tau}}=(C^{\tau_1})^{m/\tau}c^{-\de\frac{m}{\tau}}>(C^{\tau_1})^{\frac{m}{\tau}+\frac{\eps}{\tau\tau_1}}c^{-\de\frac{m}{\tau}}=C^nC^{\eps/\tau}c^{-\de\frac{m}{\tau}},
\end{align*}
and $C^{-n}c^m>1$ follows from the fact that $C^\eps>c^{\de m}$, which is equivalent to 
\[m>\frac{\eps}{\de}\log_cC.\]
The latter inequality holds, since the right-hand side is less than one, due to $\eps<\de\log_Cc$.

Hence, $C^{\tau_1}=c^\tau$. Put $C_0=C^{1/\tau}=c^{1/\tau_1}$, so that \eqref{eq:case5funceq1} and \eqref{eq:case5funceq2} become
\begin{equation}\label{eq:case5funceq1alt}
        f(t)=C_0^\tau f(t-\tau), \;\;\;\text{ for almost all }t\geq t_0-\tau_2
\end{equation}
and
\begin{equation}\label{eq:case5funceq2alt}
    f(t)=C_0^{\tau_1}f(t-\tau_1), \;\;\;\text{ for almost all }t\geq t_0.
\end{equation}
As before, let $X$ be the zero measure subset of $[t_0,\infty)$ where either \eqref{eq:case5funceq1} or \eqref{eq:case5funceq2} fails. Then, the set $\tilde{X}=X+\ZZ\tau+\ZZ\tau_1$ is also of measure zero. So, if $x\in[t_0,\infty)\sm \tilde{X}$, then for any $t\in (x+\ZZ\tau+\ZZ\tau_1)\cap[t_0,\infty)$ it holds
\[f(t)=f(x)C_0^{t-x}=K_xC_0^t,\]
where $K_x=f(x)C_0^{-x}$. If $f$ were continuous, then this would also be the formula for $f$ for all $t\geq t_0$, since $x+\ZZ\tau+\ZZ\tau_1$ is dense, i.e. $K_x$ is a constant independent of $x$. However, square integrability of $f$ yields an equally good fact, namely that $K_x$ is constant almost everywhere. This follows from Khinchin's Theorem \cite{Khinchin} which states that for $F\in L^1(\TT)$ and $\al\notin\QQ$ we have
\[\lim_{N\to\infty}\frac1N\sum_{n=1}^NF(x+n\al)=\int_0^1F(x)\di x, \;\;\;\text{ for almost all }x\in\TT.\]
We apply this Theorem with $F(x)=\abs{f(t_0+\tau_1x)}^2$. We have
\[\int_0^1F(x)\di x=\frac{1}{\tau_1}\int_{t_0}^{t_0+\tau_1}\abs{f(u)}^2\di u,\]
and the above equals for almost all $x\in[0,1]$ to
\begin{align*}
    \lim_{N\to\infty}\frac1N\sum_{n=1}^NF\bra{x+n\frac{\tau}{\tau_1}}=\lim_{N\to\infty}\frac1N\sum_{n=1}^N\abs{f(y+n\tau-m_n\tau_1)}^2,
\end{align*}
where $y=t_0+\tau_1x$ and $m_n$ is the unique integer such that $y+n\tau-m_n\tau_1\in[t_0,t_0+\tau_1)$. The right-hand side equals
\[\lim_{N\to\infty}\frac{K_y}N\sum_{n=1}^NC_0^{2(y+n\tau-m_n\tau_1)}=\frac{K_y}{\tau_1}\int_{t_0}^{t_0+\tau_1}C_0^{2u}\di u,\]
by equidistribution of the sequence $\set{y+n\tau-m_n\tau_1}_{n\in\NN}$ in the interval $[t_0,t_0+\tau_1]$ and continuity of $C_0^{2u}$. Therefore, by Khinchin's Theorem \cite{Khinchin}, there is a constant $K$ such that $K_y=K$ for almost all $y\in[t_0,t_0+\tau_1]$, which implies that 
\begin{equation}\label{eq:case5formula}
f(t)=K\cdot C_0^t, \;\;\;\text{ for almost all }t\geq t_0.    
\end{equation}

This formula cannot be valid for almost all $t\in\RR$, since $f\in L^2(\RR)$. Let $\al\in\RR$ be the infimum of all $t_0$ such that \eqref{eq:case5formula} holds. As before, let $X$ be the set of $t\in\RR$ where either \eqref{eq:precase5} fails or $t\geq\al$ and \eqref{eq:case5formula} fails. This set has measure zero, and so does the set $X'=X+\ZZ\tau_1+\ZZ\tau_2+\ZZ\tau_3$. Now, suppose that $\tau_3>\tau_1$ and let $t,t-\tau_1,t-\tau_2\geq\al$ and $t\notin X'$, so that $t-\tau_1,t-\tau_2\notin X'$ as well, and $t-\tau_3<\al$; there is actually an interval, say $I$, such that for almost all $t\in I$ these conditions are satisfied and $I-\tau_3=(\al-\eps,\al)$, for some $\eps>0$. Then, by \eqref{eq:precase5} we get that for almost all $t\in I$ we must also have $f(t-\tau_3)=K\cdot C_0^{t-\tau_3}$, contradicting the definition of $\al$. We arrive to the same conclusion if $\tau_1>\tau_3$.

If $\tau_1=\tau_3$, we solve \eqref{eq:precase5} for $f(t-\tau_1)$:
\begin{equation}\label{eq:case5solvefort1}
f(t-\tau_1)=C_0^{-\tau_1}\frac{f(t)-c_2e^{2\pi i\om_2t}f(t-\tau_2)}{1-c_2e^{2\pi i\om_2t}C_0^{-\tau_2}}.    
\end{equation}
Again, consider an interval $I$, such that for almost all $t\in I$ we have $t, t-\tau_2\geq\al$, $t, t-\tau_2\notin X'$ and $I-\tau_1=(\al-\eps,\al)$, for some $\eps>0$. Then, by \eqref{eq:case5solvefort1} for almost all $t\in I$ we have
\[f(t-\tau_1)=f(t)C_0^{-\tau_1}=K\cdot C_0^{t-\tau_1},\]
contradicting again the definition of $\al$. This concludes the proof.




\bibliographystyle{amsplain}
\bibliography{bibliography}

\end{document}